\documentclass[reqno,a4paper,twoside]{amsart}
\usepackage{enumitem}
\setenumerate{label=\textnormal{(\arabic*)}}

\allowdisplaybreaks[3]

\usepackage{amsmath,amssymb,dsfont,mathrsfs,verbatim,bm,geometry, MnSymbol}
\usepackage{microtype}
\usepackage{mathtools}
\usepackage[latin1]{inputenc}
\usepackage{array,hhline}
\usepackage[raggedright]{titlesec}
\usepackage{tikz}
\usetikzlibrary{matrix}
\usepackage{float,subfig}
\usepackage{amsrefs}

\titleformat{\section}
{\normalfont\Large\bfseries\center}{\thesection}{1em}{}
\titleformat{\subsection}
{\normalfont\large\bfseries}{\thesubsection}{1em}{}
\titleformat{\subsubsection}
{\normalfont\normalsize\bfseries}{\thesubsubsection}{1em}{}
\titleformat{\paragraph}[runin]
{\normalfont\normalsize\bfseries}{\theparagraph}{1em}{}
\titleformat{\subparagraph}[runin]
{\normalfont\normalsize\bfseries}{\thesubparagraph}{1em}{}

\titlespacing*{\chapter} {0pt}{50pt}{40pt}
\titlespacing*{\section} {0pt}{3.5ex plus 1ex minus .2ex}{2.3ex plus .2ex}
\titlespacing*{\subsection} {0pt}{3.25ex plus 1ex minus .2ex}{1.5ex plus .2ex}
\titlespacing*{\subsubsection}{0pt}{3.25ex plus 1ex minus .2ex}{1.5ex plus .2ex}
\titlespacing*{\paragraph} {0pt}{3.25ex plus 1ex minus .2ex}{1em}
\titlespacing*{\subparagraph} {\parindent}{3.25ex plus 1ex minus .2ex}{1em}

\newtheorem{theorem}{Theorem}[section]
\newtheorem{lemma}[theorem]{Lemma}
\newtheorem{proposition}[theorem]{Proposition}
\newtheorem{corollary}[theorem]{Corollary}

\theoremstyle{definition}
\newtheorem{definition}[theorem]{Definition}

\theoremstyle{remark}
\newtheorem{remark}[theorem]{Remark}

\DeclareMathOperator{\ide}{id}
\DeclareMathOperator{\End}{End}

\DeclareMathOperator{\op}{op}

\newcommand{\ot}{\otimes}

\DeclareMathAlphabet{\mathpzc}{OT1}{pzc}{m}{it}

\author{Jack Arce}
\address{Pontificia Universidad Cat\'olica del Per\'u, Secci\'on Matem\'aticas, PUCP,
Av. Universitaria 1801, San Miguel, Lima 32, Per\'u.}

\email{jarcef@pucp.edu.pe}

\thanks{ Jack Arce research was supported by PUCP-DGI-2012-0011, PUCP-DGI-2013-0036 and PUCP- Huiracocha scholarship}
\begin{document}

\title{Representations of twisted tensor products}

\maketitle
\begin{abstract}
We obtain a faithful representation of the twisted tensor product $B\otimes_{\chi} A$ of unital associative algebras,
when $B$ is finite dimensional. This  generalizes the representations of \cite{Ci} where $B=K[X]/<X^2-X>$, \cite{GGV} where $B=K[X]/<X^n>$   and \cite{JLNS} where $B=K^n$. 
Furthermore, we establish conditions to extend twisted tensor products $B\ot_{\chi} A$ and $C\ot_{\psi} A$ to a twisted tensor product $(B\times C) \ot_{\varphi} A$.
\end{abstract}

\section*{Introduction}

Let $k$ be a commutative ring and let $A$, $B$ be unitary $k$-algebras. A {\em twisted tensor product of $A$ with $B$} (over $k$) is an algebra structure defined on $A\ot_k B$, such that the canonical maps $i_A\colon A\longrightarrow  A\ot_k B$ and $i_B\colon B\longrightarrow A\ot_k B$ are algebra maps satisfying $a\ot b= i_A(a)i_B(b)$. This structure has been studied by many people with different motivations (see for instance \cite{Ca}, \cite{CSV}, \cite{GG}, \cite{GGVTP}, \cite{Ma}, \cite{Tam}, \cite{VDVK}). On one hand it is the most general solution to the problem of factorization of structures in the setting of associative algebras. Consequently, a number of examples of classical and recently defined constructions in ring theory fits into this construction. For instance, Ore extensions, skew group algebras, smash products, etcetera. On the other hand it has been proposed as the natural representative for the cartesian product of noncommutative spaces, this being based on the existing duality between the categories of algebraic affine spaces and commutative algebras, under which the cartesian product of spaces corresponds to the tensor product of algebras. Besides that, twisted tensor products arise as a useful tool for building algebras starting with simpler ones.

Given algebras $A$ and $B$, a basic problem is to determine families of twisted tensor products of $A$ with $B$ (ideally all of them) and classify them up to a natural equivalence relation. In general this is a very difficult problem. For instance, in \cite{GGV}*{Corollary~4.2} it was proven that determining all the twisted tensor products of $A$ with $k[X]/\langle X^5 \rangle$ with the form
$$
(1\ot X)(a\ot 1)= a\ot X + \gamma_1^2(a)\ot X^2 + \gamma_1^3(a)\ot X^3 + \gamma_1^4(a)\ot X^4,
$$
where  $\gamma_1^j\colon A\to A$ ($j=2,3,4$) are $k$-linear maps, is equivalent to find all the pairs $(\delta_1,\delta_2)$ of derivations of $A$, such that $[\delta_1]\cup [\delta_2]=0$, where $[\delta_i]$ is the class of $\delta_i$ n the cohomology group $H^1(A)$ and $\cup$ denotes the cup product. In spite of these difficulties, this problem was considered for different types of algebras in \cite{Ci}, \cite{GGV}, \cite{JLNS} and \cite{LNC}. When $k$ is a field and the algebras involved are finite dimensional over $k$, as in the these papers, given twisted tensor product $C$ of $A$ with $B$ and fixed a basis $\{b_1,\dots, b_n\}$, there exists maps $\gamma_i^j\colon A\to A$, satisfying suitable condition, such that
$$
\chi(1\ot a)(b_i\ot 1)= \sum_{j=1}^n  b_j\ot \gamma_i^j(a), \quad\text{for all $a\in A$ an all $i$.}
$$
In Corollary~\ref{Prop condiciones 3 y 4} and Proposition~\ref{condiciones 1 y 2} we show that these conditions are satisfied if and only if the map $\hat{\rho}_{\chi}$ and $\hat{\phi}_{\chi}$, introduced in the statements of these results, are matricial representations of algebras. The first step in the study of the twisted tensor products considered in \cite{Ci}, \cite{GGV}, \cite{JLNS} and~\cite{LNC} was determine the conditions required to the maps $\gamma^j_i$ by means of direct computations. Using the fact that $\hat{\rho}_{\chi}$  and $\hat{\phi}_{\chi}$ must be representations, these conditions arise naturally in each of the examples, as is shown in the last section.

\section{Preliminaries}
Let $K$ be a commutative ring with~$1$. All the maps considered in this notes are $K$-linear maps, the symbol denotes the tensor product over $K$, by an algebra we understand an associative unital $K$-algebra and the algebra morphisms are unital.

\subsection{Twisting maps} Let $A$ and $B$ be algebras, and let $\mu_A$ and $\mu_B$ the multiplication maps of $A$ and $B$, respectively. A {\em twisted tensor product} of  $B$ with $A$ is an algebra structure over the $K$-module $B\ot A$, such that the canonical maps
$$
i_B\colon B\longrightarrow B\ot A\quad\text{and}\quad i_A\colon A\longrightarrow B\ot A
$$
are morphisms of algebras, and $i_B(b)i_A(a)=b\ot a$ for all $b\in B$ and $a\in A$.

\smallskip

Given a twisted tensor product of $B$ with $A$, the map
$$
\chi \colon A \ot B\longrightarrow B\ot A,
$$
defined by $\chi: = \mu \circ (i_A\ot i_B)$, satisfies:

\begin{enumerate}

\smallskip

\item $\chi(1\ot b) = b\ot 1$ and $\chi(a\ot 1) = 1\ot a$, for all $a\in A$ and $b\in B$,

\smallskip

\item $\chi\circ (\mu_A\ot B) = (B\ot \mu_A)\circ (\chi\ot A)\circ (A\ot \chi)$,

\smallskip

\item $\chi \circ (A\ot \mu_B) = (\mu_B\ot A)\circ (B\ot \chi)\circ (\chi\ot B)$.

\smallskip

\end{enumerate}
A map fulfilling these conditions is called  a {\em twisting map}. Conversely, if
$$
\chi\colon A\ot B\longrightarrow B\ot A
$$
is a twisting map, then $A\ot B$ becomes a twisted tensor product via
$$
\mu_{\chi} := (\mu_B\ot\mu_A)\circ (B\ot \chi \ot A).
$$
This algebra will be denote $B\ot_{\chi} A$. It is evident that these constructions are inverse one of each other.

\begin{remark}\label{accion a der de A} The right action of $A$ on $B\ot_{\chi} A$ induced by the canonical map of $A$ into $B\ot_{\chi} A$, is the canonical right action of $A$ on $B\ot_{\chi} A$. Similarly, the left action of $B$ on $B\ot_{\chi} A$ induced by the canonical map of $B$ in $B\ot_{\chi} A$ is the canonical left action of $B$ on $B\ot_{\chi} A$.
\end{remark}

\begin{remark}\label{prop univ del Twisted tensor product} It is easy to check that the twisted tensor product $B\ot_{\chi} A$ has the following universal property: given morphisms of algebras $\varphi\colon A \to C$ and $\psi\colon B\to C$, there exists a unique morphisms of algebra $\Psi\colon B\ot_{\chi} A \longrightarrow C$ such that $\Psi\circ i_A=\varphi$ and $\Psi\circ i_B=\psi$ if and only if
$$
\mu_C\circ (\varphi\ot \psi)= \mu_C \circ (\psi\ot \varphi) \circ \chi.
$$
Consequently, if $C$ is a twisted tensor product $E\ot_{\varpi} D$, then given morphisms of algebras $f\colon A\to D$ and $g\colon B\to E$, the map
$$
g\ot f\colon B\ot_{\chi} A \longrightarrow E\ot_{\varpi} D
$$
is a morphism of algebras if and only if $\varpi \circ (f\ot g)= (g\ot f)\circ \chi$.
\end{remark}

\subsection{Twisting maps of finite dimensional algebras}\label{aptoalgdimfin} From now on we assume that $K$ is a field and that $B$ is a finite dimensional algebra over $K$. Moreover we fix a basis $\mathcal{B}:=\{b_1,\dots,b_{n}\}$ of $B$. Recall that the structure constants $\lambda^k_{ij}$ ($1\le i,j,k\le n$) of $B$ with respect to $\mathcal{B}$ are the scalars determined by the equalities
$$
b_ib_j=\sum_{k=1}^n \lambda_{ij}^k b_k.
$$
Since the multiplication $B$ is associative, the $\lambda_{ij}^k$'s satisfy
$$
\sum_{l=1}^n\lambda_{ij}^l \lambda_{lk}^m= \sum_{l=1}^n\lambda_{jk}^l \lambda_{il}^m.
$$
Moreover, since $B$ is unital, if $1=\sum_j\alpha_j b_j$, then
$$
\sum_{j=1}^n\alpha_j\lambda_{ji}^k=\sum_{j=1}^n\alpha_j\lambda_{ij}^k=\delta_{ki}.
$$

A  map $\chi\colon A \otimes B \longrightarrow B\ot A$ determine unique maps $\gamma_i^j\colon A\longrightarrow A$ ($0\le i,j<n$),
such that
\begin{equation}\label{eq def de los gamma ij}
\chi(a\ot b_i)= \sum_{j=1}^n  b_j\ot \gamma_i^j(a), \quad\text{for all $i$.}
\end{equation}

\begin{proposition}\label{condiciones para ser torcimiento} The map $\chi$ is a twisting map if and only if the functions $\gamma^i_j$ satisfy the following properties:

\begin{enumerate}

\smallskip

\item $\gamma_i^j(1)=\delta_{ij} 1$,

\smallskip

\item $\gamma_i^k(aa')=\sum_j\gamma_j^k(a)\gamma_i^j(a')$,

\smallskip

\item $\alpha_k\ide=\sum_i\alpha_i\gamma_i^k$,

\smallskip

\item $\sum_k \lambda_{ij}^k\,\gamma_k^m=\sum_k\sum_l \lambda_{kl}^m\, \gamma^l_j\circ \gamma^k_i$.

\end{enumerate}
More precisely, conditions (1) and (2) are satisfied if and only if $\chi$ is compatible with the algebra structure of $A$, and conditions (3) and (4) are satisfied if and only if it is compatible with that of~$B$.
\end{proposition}

\begin{proof} Left to the reader.
\end{proof}

\section{The canonical representation}
\label{The canonical representation}
Recall that the regular left representation of a $K$-algebra $C$ is the morphism $l\colon  C \longrightarrow \End_K(C)_C$ defined by $l(c):=l_c$,
where $l_c\in \End_K(C)$ is the map given by $l_c(d):=cd$ for all $d\in C$. Since $l_c(1)=c$ for all $c\in C$, this is a faithful representation. Consequently, for each algebra $L$ the map $l^L\colon B^{\op}\ot L \longrightarrow \End(B^{\op}\ot L)_L$, given by
$$
\begin{tikzpicture}
\draw [->]   (0,-0.2) -- (2,-0.2); \draw (-0.8,-0.2) node {$B^{\op}\ot L$}; \draw (1,0) node {$l^L$}; \draw (3.4,-0.2) node {$\End_K(B^{\op}\ot L)_L$}; \draw (4.8,-0.2) node {,};
\draw [->] (-0.6,-0.5) -- (0.9,-2); \draw [>->] (1.1,-2) -- (2.6,-0.5);
\draw (-0.3,-1.25) node {$l$}; \draw (2.2,-1.25) node {$i$};
\draw (1,-2.3) node {$\End_K(B^{\op}\ot L)_{B^{\op}\ot L}$};
\end{tikzpicture}
$$
in which $i$ is the canonical inclusion, is a faithful representation, called the {\em canonical representation} of $B^{\op}\ot E$.

\smallskip

From now on we set $E:=\End_K(A)$. In the rest of the section we assume that we are in the conditions of Subsection~\ref{aptoalgdimfin} and we denote by $\mathcal{B}'$ the basis $\{b_1^{\op},\cdots,b_n^{\op}\}$ of~$B^{\op}$.

\begin{proposition}\label{Prop condiciones 3 y 4} Items (3) and (4) of Proposition~\ref{condiciones para ser torcimiento} are satisfied if and only if
$$
\begin{tikzpicture}
\draw [->]   (0.2,-0.2) -- (2,-0.2); \draw (-0.2,-0.2) node {$B^{\op}$}; \draw (1.05,0) node {$\rho_{\chi}$}; \draw (2.75,-0.2) node {$B^{\op}\ot E$};
\draw (-0.2,-0.6) node {$b_k^{\op}$}; \draw [|->]   (0.2,-0.6) -- (1.8,-0.6); \draw (2.8,-0.6) node {$\sum_{j} b_j^{\op}\ot \gamma_k^j$};
\end{tikzpicture}
$$
is a morphism of $K$-algebras.
\end{proposition}

\begin{proof} This follows by a direct computation.
\end{proof}

\begin{corollary}\label{Coro condiciones 3 y 4}
Consider the right $E$-linear bijection
$$
\begin{tikzpicture}
\draw [->]   (0.6,0) -- (2,0); \draw (-0.2,0) node {$B^{\op}\ot E$}; \draw (1.3,0.2) node {$\xi_{\mathcal{B}'}$}; \draw (2.3,0.) node {$E^n$};
\draw (-0.2,-0.4) node {$b_k^{\op}\ot 1$}; \draw [|->]   (0.45,-0.4) -- (2,-0.4); \draw (2.3,-0.4) node {$e_k$};
\draw (2.6,-0.1) node {,};
\end{tikzpicture}
$$
where $\{e_1,\dots,e_{n}\}$ is the canonical basis of $E^n$. Items (3) and (4) of  Proposition~\ref{condiciones para ser torcimiento} are satisfied if and only if the map
$$
\begin{tikzpicture}
\draw [->]   (0.2,0) -- (2.4,0); \draw (-0.2,0) node {$B^{\op}$}; \draw (1.3,0.24) node {$\dddot{\rho}_{\!\chi}$}; \draw (3.4,0) node {$\End(E^n)_E$};
\draw (-0.2,-0.4) node {$b^{\op}$}; \draw [|->]   (0.2,-0.4) -- (1.7,-0.4); \draw (3.5,-0.4) node {$\xi_{\mathcal{B}'}\circ l^E(\rho_{\chi}(b^{\op}))\circ \xi^{-1}_{\mathcal{B}'}$};
\end{tikzpicture}
$$
is a representation of $B^{\op}$ or, equivalently, if the map $\hat{\rho}_{\chi}\colon B^{\op}\longrightarrow M_n(E)$, that sends $b^{\op}\in B^{\op}$ to the matrix of $\dddot{\rho}_{\!\chi}(b^{\op})$ with respect to the canonical basis, is a matrix representation.
\end{corollary}

\begin{proof} Since $l^E$ is faithful and $\xi_{\mathcal{B}'}$ is a bijective map, this is an immediate consequence of Proposition~\ref{Prop condiciones 3 y 4}.
\end{proof}

\begin{definition}\label{matrices de estructura} By definition, for each $b_k\in \mathcal{B}$ {\em the structure matrix of $b_k$ with respect to $(B,\mathcal{B})$} is the matrix $[b_k]_{\mathcal{B}}:= \left(\lambda_{kj}^i\right)_{ij}$, where the $\lambda_{uv}^w$'s ($1\le u,v,w \le n$) are the structure constants of $B$ with respect to~$\mathcal{B}$.
\end{definition}

\begin{remark} A direct computation shows that
$$
\hat{\rho}_{\chi}(b_k^{\op})=\begin{pmatrix} \displaystyle{\sum_{l=1}^n} \lambda^1_{1l}\gamma^l_k & \dots & \displaystyle{\sum_{l=1}^n} \lambda^1_{nl}\gamma^l_k \\
                              \vdots & \ddots & \vdots\\
                              \displaystyle{\sum_{l=1}^n} \lambda^n_{1l}\gamma^l_k & \dots & \displaystyle{\sum_{l=1}^n} \lambda^n_{nl}\gamma^l_k\end{pmatrix}
                              =\sum_{j=1}^n \gamma^j_k[b^{\op}_j]_{\mathcal{B}}\quad\text{for all $k$,}
$$
where $[b^{\op}_j]_{\mathcal{B}}$ is the structure matrix of $b^{\op}_j$ with respect to $(B^{\op},\mathcal{B})$.
\end{remark}

\begin{proposition}\label{condiciones 1 y 2} Conditions~(1) and~(2) of Proposition~\ref{condiciones para ser torcimiento} are fulfilled if and only if the function
$
\phi_{\chi}\colon A \longrightarrow \End(B\ot A)_A
$,
defined by $\phi_{\chi}(a)(b_k\ot 1):= \sum_j b_j\ot \gamma^j_k(a)$, is a morphism of algebras or, equivalently, if the map $\hat{\phi}_{\chi}\colon A \longrightarrow M_n(A)$, that sends $a\in A$ to the matrix of $\phi_{\chi}(a)$ with respect to the basis $\mathcal{B}\ot 1:=\{b_1\ot 1,\dots, b_n\ot 1 \}$ of $B\ot A$, is a matrix representation.
\end{proposition}

\begin{proof} This follows by a direct computation.
\end{proof}

\begin{remark}\label{phi sub chi es inyectiva} From equality~\eqref{eq def de los gamma ij} it follows that $\phi_{\chi}(a)(b\ot 1) = \chi(a\ot b)$ for all $b\in B$. In particular~$\phi_{\chi}(a)(1\ot 1)=1\ot a$, which implies that $\phi_{\chi}$ is injective.
\end{remark}

\begin{corollary} \label{Coro coniciones 1 y 2 en presencia de cond 4} Consider the isomorphisms of right $A$-modules
$$
\begin{tikzpicture}
\draw [->]   (0.4,0) -- (2.3,0); \draw (-0.2,0) node {$B\ot A$}; \draw (1.3,0.3) node {$\grave{\xi}_{\mathcal{B}}$}; \draw (2.6,0) node {$A^n$}; \draw (2.9,-0.1) node {,};
\draw (-0.2,-0.4) node {$b_k\otimes 1$}; \draw [|->]   (0.4,-0.4) -- (2.3,-0.4); \draw (2.6,-0.4) node {$f_k$};
\end{tikzpicture}
$$
where $\{f_1,\dots,f_{n}\}$ is the canonical basis of $A^n$. If $\chi$ is a twisting map, then the function
$$
\begin{tikzpicture}
\draw [->]   (0.3,0) -- (2.4,0); \draw (-0.4,0) node {$B\ot_{\chi} A$}; \draw (1.3,0.3) node {${}_{\chi}l^A_{\mathcal{B}}$}; \draw (3.4,0) node {$\End(A^n)_A$};
\draw (-0.4,-0.4) node {$b\ot a$}; \draw [|->]   (0.3,-0.4) -- (1.7,-0.4); \draw (3.5,-0.4) node {$\grave{\xi}_{\mathcal{B}}\circ l_{b\ot 1}\circ \phi_{\chi}(a)\circ \grave{\xi}_{\mathcal{B}}^{-1}$};
\end{tikzpicture}
$$
is a representation of $B\ot_{\chi} A$ in $\End(A^n)_A$.
\end{corollary}

\begin{proof} This follows from Proposition~\ref{condiciones 1 y 2}, the fact that under condition~(4) of Proposition~\ref{condiciones para ser torcimiento}
$$
\phi_{\chi}(a)\circ l_{b_i\ot 1} =\sum_{j}{l_{b_j\ot 1}\circ\phi_{\chi}(\gamma^j_i(a))},
$$
and the fact that $\xi_{\mathcal{B}'}$ is a bijective map.
\end{proof}

\begin{remark} A direct computation shows that
$$
l_{b\ot 1}\circ \phi_{\chi}(a)(b'\ot 1)=(b\ot 1)\chi(a\ot b')\quad\text{for all $a\in A$ and $b,b'\in B$.}
$$
Consequently $l_{b\ot 1}\circ \phi_{\chi}(a)(1\ot 1)=b\ot a$, which implies that ${}_{\chi}l^A_{\mathcal{B}}$ is an injective map.
\end{remark}

\begin{corollary}\label{la otra representacion matricial} If $\chi$ is a twisting map, then the formulas
$$
\varphi_{\chi}(a):= \begin{pmatrix} \gamma^1_1(a) & \dots & \gamma^1_n(a) \\ \vdots & \ddots & \vdots \\ \gamma^n_1(a) & \dots & \gamma^n_n(a)  \end{pmatrix}  \quad\text{and}\quad \varphi_{\chi}(b_k):=\begin{pmatrix} \lambda^1_{k1} \cdot 1_A & \dots & \lambda^1_{kn} \cdot 1_A\\ \vdots & \ddots & \vdots \\ \lambda^n_{k1} \cdot 1_A& \dots & \lambda^n_{kn} \cdot 1_A\end{pmatrix},
$$
for all $a\in A$ and $1\le k\le n$, define a faithful representation $\varphi_{\chi}\colon B\ot_{\chi} A\longrightarrow M_n(A)$.
\end{corollary}

\begin{proof} It suffices to note that the matrices of ${}_{\chi}l^A_{\mathcal{B}}(1\ot a)$ and
${}_{\chi}l^A_{\mathcal{B}}(b_k\ot 1)$ in the canonical basis of $A^n$, are the matrices $\varphi_{\chi}(a)$ and $\varphi_{\chi}(b_k)$, respectively.
\end{proof}

\section{Extensions of twisting maps}\label{seccionextension}
Let $A$, $B$ and $C$ be $K$-algebras with $B$ and $C$ finite dimensional. Write $D:=B\times C$. Let $i_B\colon B \to D$, $i_C\colon C \to D$, $p_B\colon D \to B$ and $p_C\colon D \to C$ be the canonical maps. In this section we study on one hand twisting maps $\psi\colon A\ot D\longrightarrow D\ot A$ such that the map $\Theta:=(p_B\ot A)\circ\psi\circ (A\ot i_B)$ is a twisting map, and on the other hand twisting maps such that both maps, $\Theta$ and $\Upsilon:= (p_C\ot A)\circ\psi\circ (A\ot i_C)$, are twisting maps.

Fix basis $\mathcal{B}=\{b_1,b_2,\hdots,b_n\}$ and $\mathcal{C}:= \{c_1,\dots,c_m\}$ of $B$ and $C$, respectively, and let $\mathcal{D}$ denote the ordered basis $\{b_1,b_2,\hdots,b_n,c_1,\dots,c_m\}$ of $\mathcal{D}$, where we identify $b_i$ with $(b_i,0)$ and $c_i$ with $(0,c_i)$. So, $\mathcal{D}=\{d_1,\dots,d_{m+n}\}$, where
$$
d_i:=\begin{cases} b_i &\text{if $i\le n$,} \\ c_{i-n} &\text{if $i > n$.}
\end{cases}
$$
The structure matrices of the elements of $D$ with respect to $(D,\mathcal{D})$ are
$$
\left[b_k\right]_{\mathcal{D}} =\begin{pmatrix} \left[b_k\right]_{\mathcal{B}} & 0 \\ 0 & 0 \end{pmatrix} \text{ ($1\le k \le n$)} \qquad\text{and}\qquad \left[c_k\right]_{\mathcal{D}} =\begin{pmatrix} 0 & 0 \\ 0 & \left[c_l\right]_{\mathcal{C}}  \end{pmatrix} \text{ ($1\le l \le m$)}.
$$
Let $\alpha_1,\dots,\alpha_n,\beta_1,\dots,\beta_m$ be the scalars such that
$$
1_D = (1_B, 1_C) =\sum_{k=1}^n\alpha_k b_k +\sum_{l=1}^m \beta_l c_l.
$$

\smallskip

\noindent Each map $\psi\colon A\ot D \longrightarrow D\ot A$ determines uniquely functions $\tilde{\gamma}_i^j$ ($1\le i,j < m+n$), such that
\begin{equation}\label{ecuaciongammatilde}
\psi(a\ot d_i)= \sum_{j=1}^{m+n} d_j \ot \tilde{\gamma}_i^j(a)
\end{equation}
Let $\lambda^k_{ij}$ ($1\le i,j,k\le n$) and $\eta^k_{ij}$ ($1\le i,j,k\le m$) be the structure constants of $B$ with respect to $\mathcal{B}$ and of $C$ with respect to $\mathcal{C}$, respectively.
Write
\allowdisplaybreaks
\begin{align*}
 & B^{(1)}_k:=\begin{pmatrix} \displaystyle{\sum_{l=1}^n} \lambda^1_{1l}\tilde{\gamma}^l_k & \dots & \displaystyle{\sum_{l=1}^n} \lambda^1_{nl}\tilde{\gamma}^l_k \\
                              \vdots & \ddots & \vdots\\
                              \displaystyle{\sum_{l=1}^n} \lambda^n_{1l}\tilde{\gamma}^l_k & \dots & \displaystyle{\sum_{l=1}^n} \lambda^n_{nl}\tilde{\gamma}^l_k
\end{pmatrix}
\in M_n(E) && \text{($1\le k \le n$),}\\[1.5\jot]
 & B^{(2)}_k:=\begin{pmatrix} \displaystyle{\sum_{l=1}^m} \eta^1_{1l}\tilde{\gamma}^{l+n}_k & \dots & \displaystyle{\sum_{l=1}^m} \eta^1_{ml}\tilde{\gamma}^{l+n}_k \\
                              \vdots & \ddots & \vdots\\
                              \displaystyle{\sum_{l=1}^m} \eta^m_{1l}\tilde{\gamma}^{l+n}_k & \dots & \displaystyle{\sum_{l=1}^m} \eta^m_{ml}\tilde{\gamma}^{l+n}_k
\end{pmatrix}
\in M_m(E) && \text{($1\le k \le n$),}\\[1.5\jot]
 & C^{(1)}_k:=\begin{pmatrix} \displaystyle{\sum_{l=1}^n} \lambda^1_{1l}\tilde{\gamma}^l_{k+n} & \dots & \displaystyle{\sum_{l=1}^n} \lambda^1_{nl}\tilde{\gamma}^l_{k+n} \\
                              \vdots & \ddots & \vdots\\
                              \displaystyle{\sum_{l=1}^n} \lambda^n_{1l}\tilde{\gamma}^l_{k+n} & \dots & \displaystyle{\sum_{l=1}^n} \lambda^n_{nl}\tilde{\gamma}^l_{k+n}
\end{pmatrix}
\in M_n(E) && \text{($1\le k \le m$)}
\intertext{and}
 & C^{(2)}_k:=\begin{pmatrix} \displaystyle{\sum_{l=1}^m} \eta^1_{1l}\tilde{\gamma}^{l+n}_{k+n} & \dots & \displaystyle{\sum_{l=1}^m} \eta^1_{ml}\tilde{\gamma}^{l+n}_{k+n} \\
                              \vdots & \ddots & \vdots\\
                              \displaystyle{\sum_{l=1}^m} \eta^m_{1l}\tilde{\gamma}^{l+n}_{k+n} & \dots & \displaystyle{\sum_{l=1}^m} \eta^m_{ml}\tilde{\gamma}^{l+n}_{k+n}
\end{pmatrix}
\in M_m(E) && \text{($1\le k \le m$).}
\end{align*}
Suppose that $\psi$ is a twisting map. The representation $\hat{\rho}_{\psi}\colon D^{\op} \longrightarrow M_{m+n}(E)$ of  Corollary~\ref{Coro condiciones 3 y 4} is given by
\begin{align}
&\hat{\rho}_{\psi}(b^{\op}_k)=\begin{pmatrix} B^{(1)}_k & 0\\ 0 & B^{(2)}_k\end{pmatrix} \text{ $(1 \le k \le n)$}\label{rhohatsubpsidelosb}
\shortintertext{and}
&\hat{\rho}_{\psi}(c^{\op}_k)=\begin{pmatrix} C^{(1)}_k & 0\\ 0 & C^{(2)}_k\end{pmatrix} \text{ $(1 \le k \le m)$.}\label{rhohatsubpsidelosc}
\end{align}
\begin{lemma}\label{lema para citar}
The map $\hat{\rho}_{\psi}\colon D^{\op} \longrightarrow M_{m+n}(E)$ defined by~(3.3)-(3.4) is a representation if and only if the following conditions are satisfied:
\begin{align}
& B_i^{(l)} B_j^{(l)}=\sum_{k=1}^n \lambda^k_{ji}B_k^{(l)} &&\text{for $l\in \{1,2\}$ and  $1\le i,j\le n$,}\label{equno}\\
& C_i^{(l)} C_j^{(l)}=\sum_{k=1}^m \eta^k_{ji}C_k^{(l)} &&\text{for $l\in \{1,2\}$ and  $1\le i,j\le m$,}\\
& B_i^{(l)} C_j^{(l)}= C_j^{(l)} B_i^{(l)} =0 &&\text{for $l\in \{1,2\}$, $1\le i\le n$ and $1\le j \le m$,}\\
& \sum_{i=1}^n \alpha_i B_i^{(1)}+\sum_{j=1}^m \beta_j C_j^{(1)}=\ide_{M_n(E)}&&\label{otra mas}\\
\shortintertext{and}
& \sum_{i=1}^n \alpha_i B_i^{(2)}+\sum_{j=1}^m \beta_j C_j^{(2)}=\ide_{M_m(E)}.\label{eqdos}
\end{align}
\end{lemma}

\begin{proof}
This follows from a direct computation.
\end{proof}

The representation $\hat{\phi}_{\psi}\colon A \longrightarrow M_{m+n}(A)$ of Proposition~\ref{condiciones 1 y 2} is given by
\begin{equation}\label{representacion phi para D}
\hat{\phi}_{\psi}(a)=
\begin{pmatrix}
\tilde{\gamma}^{1}_{1}(a) & \dots & \tilde{\gamma}^{1}_{n+m}(a) \\
\vdots & \ddots & \vdots \\
\tilde{\gamma}^{n+m}_{1}(a) & \dots & \tilde{\gamma}^{n+m}_{n+m}(a)
\end{pmatrix}
\quad
\text{for $a\in A$,}
\end{equation}
Write
$$
\hat{\phi}_{\chi}(a)=
\begin{pmatrix}
\Gamma^0_0(a) & \Gamma^0_1(a) \\
\Gamma^1_0(a) & \Gamma^1_1(a)
\end{pmatrix},
$$
with $\Gamma^0_0(a)\in M_n(K)$, $\Gamma^0_1(a)\in M_{n\times m}(K)$, $\Gamma^1_0(a)\in M_{m\times n}(K)$ and $\Gamma^1_1(a)\in M_m(K)$.

\begin{remark}\label{significado de Gamma^0_1=0} Equality~\eqref{ecuaciongammatilde} shows that $\Gamma^0_1=0$ if and only if $\psi(A\ot C)\subseteq C\ot A$.
\end{remark}

\begin{lemma}\label{lema que faltaba}
The map $\hat{\phi}_{\psi}\colon A \longrightarrow M_{m+n}(A)$ defined by~\ref{representacion phi para D}, is a representation if and only if the following conditions are satisfied:
\begin{align}
&\Gamma^0_0(1)=I_n,\quad \Gamma^1_1(1)=I_m,\quad \Gamma^0_1(1)=0, \quad \Gamma^1_0(1)=0\label{eqtres}
\shortintertext{and}
&\Gamma^p_q(aa')=\Gamma^p_0(a)\Gamma^0_q(a')+\Gamma^p_1(a)\Gamma^1_q(a') \quad \text{for $a,a'\in A$ and $0\le p,q\le 1$.}\label{eqcuatro}
\end{align}
\end{lemma}

\begin{proof}
Straightforward.
\end{proof}

\begin{lemma}\label{extension01}
Let $\psi\colon A\ot D \longrightarrow D\ot A$ be a map and let $\tilde{\gamma}_i^j$ ($1\le i,j \le m+n$) be the maps determined
by the equality~\eqref{ecuaciongammatilde}.
Let $i_B\colon B\to D$ and $p_B\colon D\to B$ be the canonical inclusion and the canonical surjection, respectively,
and suppose that the map $\Theta:=(p_B\ot A)\circ\psi\circ (A\ot i_B)$ is a twisting map.
Then $\psi$ is a twisting map if and only if the following conditions are satisfied
\begin{align}
& B_i^{(2)} B_j^{(2)}=\sum_{k=1}^n \lambda^k_{ji}B_k^{(2)} &&\text{for $1\le i,j\le n$,}\label{primera}\\
& C_j^{(1)}=0,\\
& C_i^{(2)} C_j^{(2)}=\sum_{k=1}^m \eta^k_{ji}C_k^{(2)} &&\text{for $1\le i,j\le m$,}\label{eqaa}\\
& B_i^{(2)} C_j^{(2)}= C_j^{(2)} B_i^{(2)} =0 &&\text{for $1\le i\le n$ and $1\le j \le m$,}\label{eqab}\\
& \sum_{i=1}^n \alpha_i B_i^{(2)}+\sum_{j=1}^m \beta_j C_j^{(2)}=\ide_{M_m(E)},\label{eqac}\\
& \Gamma^p_1(aa')=\Gamma^p_0(a)\Gamma^0_1(a')+\Gamma^p_1(a)\Gamma^1_1(a')&&\text{for $a,a'\in A$ and $0\le p\le 1$,}\label{eqad}\\
& \Gamma^1_0(aa')=\Gamma^1_0(a)\Gamma^0_0(a')+\Gamma^1_1(a)\Gamma^1_0(a')&&\text{for $a,a'\in A$,}\label{eqae}\\
& \Gamma^0_1(a)\Gamma^1_0(a')=0&&\text{for $a,a'\in A$}\label{eqaf}
\shortintertext{and}
&\Gamma^1_1(1)=I_m,\quad \Gamma^0_1(1)=0, \quad \Gamma^1_0(1)=0.\label{eqag}
\end{align}
\end{lemma}

\begin{proof} By Corollary~\ref{Coro condiciones 3 y 4}, Proposition~\ref{condiciones 1 y 2},  Lemma~\ref{lema para citar} and Lemma~\ref{lema que faltaba}, we know that~$\psi$ is a twisting map if and only if conditions~\eqref{equno}-\eqref{eqdos}, \eqref{eqtres} and \eqref{eqcuatro} are fulfilled. Since $\Gamma^0_0(a)=\hat{\phi}_{\Theta}(a)$ and $\Theta$ is a twisting map, we have
$$
\Gamma^0_0(aa')=\Gamma^0_0(a)\Gamma^0_0(a')\quad\text{and}\quad\Gamma^0_0(1)=I_n.
$$
Hence conditions~\eqref{eqtres} and~\eqref{eqcuatro} are equivalent to conditions~\eqref{eqad}-\eqref{eqag}. Moreover, again since $\Theta$ is a twisting map and $B_i^{(1)}=\hat{\rho}_{\Theta}(b^{op}_i)$, we have
$$
B_i^{(1)} B_j^{(1)}=\sum_{k=1}^n \lambda^k_{ji}B_k^{(1)}\quad\text{and}\quad\sum_{i=1}^n \alpha_i B_i^{(1)}=\ide_{M_n(E)}.
$$
So, condition~\eqref{equno} is satisfied for $l=1$, and condition~\eqref{otra mas} becomes equivalent to
$$
\sum_{j=1}^m \beta_j C_j^{(1)}=0.
$$
Consequently, if $\psi$ is a twisting map, then
$$
\hat{\rho}_{\psi}(1_C)=\sum_{j=1}^{m} \beta_j\hat{\rho}_{\psi}(c_j) =\sum_{j=1}^{m} \beta_j \begin{pmatrix} 0& 0\\ 0 & C^{(2)}_j\end{pmatrix},
$$
since $\hat{\rho}_{\psi}$ is a representation, and so~$C^{(1)}_j=0$. Conversely, if conditions~\eqref{primera}-\eqref{eqag} are fulfilled and $\Theta$ is a twisting map, then a direct computation shows that conditions~\eqref{equno}-\eqref{eqdos}, \eqref{eqtres} and \eqref{eqcuatro} are satisfied, which proves that $\psi$ is a twisting map
\end{proof}

\begin{proposition}\label{teorema fundamental sobre extensiones}
Let $\psi\colon A\ot D \longrightarrow D\ot A$ be a map and let $\tilde{\gamma}_i^j$ ($1\le i,j \le m+n$) be the maps determined
by the equality~\eqref{ecuaciongammatilde}.
Let $i_B\colon B\to D$ and $p_B\colon D\to B$ be the canonical inclusion and the canonical surjection, respectively,
and suppose that the map $\Theta:=(p_B\ot A)\circ\psi\circ (A\ot i_B)$ is a twisting map.
Then $\psi$ is a twisting map if and only if the following conditions are satisfied
\begin{align}
& B_i^{(2)} B_j^{(2)}=\sum_{k=1}^n \lambda^k_{ji}B_k^{(2)} &&\text{for $1\le i,j\le n$,}\label{primera1}\\
& C_j^{(1)}=0,\\
& C_i^{(2)} C_j^{(2)}=\sum_{k=1}^m \eta^k_{ji}C_k^{(2)} &&\text{for $1\le i,j\le m$,}\label{eqaa1}\\
& B_i^{(2)} C_j^{(2)}= C_j^{(2)} B_i^{(2)} =0 &&\text{for $1\le i\le n$ and $1\le j \le m$,}\label{eqab1}\\
& \sum_{i=1}^n \alpha_i B_i^{(2)}+\sum_{j=1}^m \beta_j C_j^{(2)}=\ide_{M_m(E)},\label{eqac1}\\
& \Gamma^0_1=0,\\
& \Gamma^1_1(aa')=\Gamma^1_1(a)\Gamma^1_1(a')&&\text{for $a,a'\in A$,}\label{eqad1}\\
& \Gamma^1_0(aa')=\Gamma^1_0(a)\Gamma^0_0(a')+\Gamma^1_1(a)\Gamma^1_0(a')&&\text{for $a,a'\in A$,}\label{eqae1}\\
\shortintertext{and}
&\Gamma^1_1(1)=I_m, \quad \Gamma^1_0(1)=0.\label{eqag1}
\end{align}
\end{proposition}

\begin{proof} $\Leftarrow$)\enspace By Lemma~\ref{extension01}, this is clear.

\smallskip

\noindent $\Rightarrow$) Again by Lemma~\ref{extension01} it suffices to check that $\Gamma^0_1=0$ or, equivalently, that $\psi(A\ot C)\subseteq C\ot A$. Module a basis change we can assume that $b_1=1_B$. In this case, since $C_i^{(1)}$ we have:
$$
\tilde{\gamma}_{i+n}^j =\sum_{l=1}^n \delta_{jl}\tilde{\gamma}_{i+n}^l =\sum_{l=1}^n \lambda_{1l}^j\tilde{\gamma}_{i+n}^l = \bigl(C_i^{(1)}\bigr)_{_{j1}}=0
$$
for $i\in \{1,\dots,n\}$ and $j\in \{1,\dots,m\}$, as desired.
\end{proof}

\begin{remark}\label{representacion extension 1}
For each $a\in A$, write
\begin{align*}
&\varphi_B(a):= \begin{pmatrix} \tilde{\gamma}^1_1(a) & \dots & \tilde{\gamma}^1_n(a) \\ \vdots & \ddots & \vdots \\ \tilde{\gamma}^n_1(a) & \dots & \tilde{\gamma}^n_n(a)  \end{pmatrix},\\[1.5\jot]
&\varphi_C(a):= \begin{pmatrix} \tilde{\gamma}^{n+1}_{n+1}(a) & \dots & \tilde{\gamma}^{n+1}_{n+m}(a) \\ \vdots & \ddots & \vdots \\ \tilde{\gamma}^{n+m}_{n+1}(a) & \dots & \tilde{\gamma}^{n+m}_{n+m}(a)  \end{pmatrix}
\shortintertext{and}
&\Delta(a):= \begin{pmatrix} \tilde{\gamma}^{n+1}_1(a) & \dots & \tilde{\gamma}^{n+1}_n(a) \\ \vdots & \ddots & \vdots \\ \tilde{\gamma}^{n+m}_1(a) & \dots & \tilde{\gamma}^{n+m}_n(a)  \end{pmatrix}.
\end{align*}
If the hypothesis of Proposition~\ref{teorema fundamental sobre extensiones} are satisfied, then
\begin{equation}\label{matrix extension 1}
\varphi_{\psi}(a)=\begin{pmatrix}
\varphi_B(a) & 0\\ \Delta(a) & \varphi_C(a)
\end{pmatrix}
\quad\text{for all $a\in A$.}
\end{equation}
Furthermore, since $\varphi_{\psi}$ is a representation, the following equalities hold
\begin{align*}
&\varphi_B(aa')=\varphi_B(a)\varphi_B(a'),\\
&\varphi_C(aa')=\varphi_C(a)\varphi_C(a')
\shortintertext{and}
&\Delta(aa')= \Delta(a)\varphi_B(a') + \varphi_C(a)\Delta(a'),
\end{align*}
for all $a,a'\in A$.
\end{remark}

\begin{corollary}\label{extension02} Suppose that the hypothesis of Proposition~\ref{extension01} are satisfied.
Let $i_C\colon C\to D$ and $p_C\colon D\to C$ be the canonical inclusion and the canonical surjection, respectively. If the map $\Upsilon:=(p_C\ot A)\circ\psi\circ (A\ot i_C)$ is a twisting map. Then $\psi$ is a twisting map if and only if $\psi=\Theta \oplus \Upsilon$
\end{corollary}

\begin{proof} Use Proposition~\ref{teorema fundamental sobre extensiones} and Remark~\ref{significado de Gamma^0_1=0} applied to $B$ and $C$.
\end{proof}

\section{Change of Basis}
The representations introduced in Section~\ref{The canonical representation} depend on the choice of the basis in $B$. In this section we analyze how the representations $\varphi_{\chi}$ and $\hat \rho_\chi$ behave under a base change.

Let $A$, $B$ and $C$ algebras over a field $K$ and let $f\colon B\to C$ be a morphism of algebras. Assume that $B$ and $C$ are finite dimensional and fix basis $\mathcal{B}=\{b_1,\dots,b_n\}$ and $\mathcal{C}=\{c_1,\dots,c_m\}$ of $B$ and $C$, respectively. Consider the matrix
$$
M_{\mathcal{B}}^{\mathcal{C}}(f):=\begin{pmatrix}
                    \zeta_1^1\cdot 1_A & \dots & \zeta_n^1\cdot 1_A\\
                    \vdots & \ddots & \vdots \\
                    \zeta_1^m\cdot 1_A & \dots & \zeta_n^m\cdot 1_A\\
   \end{pmatrix}
\in M_{m\times n}(A),
$$
where the scalars $\zeta_i^j\in K$ are determined by the equalities
$$
f(b_j)= \sum_{i=1}^m \zeta^i_j \cdot c_i.
$$

\begin{proposition}\label{morfismo inducido} Let $\chi\colon A\ot B \longrightarrow B\ot A$ and $\varpi \colon A\ot C \longrightarrow C\ot A$ be twisting maps. The map
$$
f\ot A\colon B\ot_{\chi} A \longrightarrow C\ot_{\varpi} A
$$
is a morphism of algebras if and only if
\begin{equation}\label{conparamorfusandomatr}
\varphi_{\varpi}(a) \  M_{\mathcal{B}}^{\mathcal{C}}(f)= M_{\mathcal{B}}^{\mathcal{C}}(f) \ \varphi_{\chi}(a)\quad\text{for all $a\in A$,}
\end{equation}
where $\varphi_{\varpi}$ and $\varphi_{\chi}$ are the representations defined in Corollary~\ref{la otra representacion matricial}.
\end{proposition}

\begin{proof} Let $\gamma_i^j \colon A\to A$ ($1\le i,j \le n$) and $\tilde{\gamma}_i^j \colon A\to A$ ($1\le i,j \le m$) be the maps determined by the equality~\eqref{eq def de los gamma ij} applied to $\chi$ and $\varpi$, respectively.  A direct computation using Remark~\ref{prop univ del Twisted tensor product} shows that $f\ot A$ is a morphism if and only if
\begin{equation}\label{con los gammas}
\sum_{i=1}^n \zeta_i^j \cdot \gamma_k^i(a) = \sum_{u=1}^m \zeta_k^u\cdot \tilde{\gamma}_u^j(a)\quad\text{for all $a\in A$.}
\end{equation}
But this happens if and only if condition~\eqref{conparamorfusandomatr} is fulfilled.
\end{proof}
\begin{remark}\label{rho hat por morfismos}
Let
$$
\widehat{M}_{\mathcal{B}}^{\mathcal{C}}(f):=\begin{pmatrix}
                    \zeta_1^1\cdot \ide & \dots & \zeta_n^1\cdot \ide\\
                    \vdots & \ddots & \vdots \\
                    \zeta_1^m\cdot \ide & \dots & \zeta_n^m\cdot \ide\\
   \end{pmatrix}
\in M_{m\times n}(E).
$$
A direct computation using~\eqref{conparamorfusandomatr} shows that under
the hypothesis of Proposition~\ref{morfismo inducido}, we have
$$
\hat{\rho}_{\varpi}(f(b)^{op})\ \widehat{M}_{\mathcal{B}}^{\mathcal{C}}(f) = \widehat
M_{\mathcal{B}}^{\mathcal{C}}(f) \ \hat{\rho}_{\chi}(b^{op}),
$$
for all $b\in B$.
\end{remark}

\begin{corollary}\label{cambio de base} Let $\mathcal{B}':=\{b'_1,\dots,b'_n\}$ be another basis of $B$ and let $\varphi'_{\chi}$ and $\hat{\rho}_{\chi}'$ be
the representations associated to $\chi$ according Corollaries~\ref{Coro condiciones 3 y 4} and~\ref{la otra representacion matricial}, but using the basis
$\mathcal{B}'$ instead of $\mathcal{B}$. Let $M:=M_{\mathcal{B}}^{\mathcal{B}'}(\ide_B)$ and
$\widehat M:=\widehat M_{\mathcal{B}}^{\mathcal{B}'}(\ide_B)$. Then
$$
\varphi'_{\chi}(a)= M \  \varphi_{\chi}(a)\   M^{-1}\quad\text{and}\quad
\hat{\rho}_{\chi}'(b^{op}) = \widehat
M \ \hat{\rho}_{\chi}(b^{op})\ \widehat{M}^{-1},\text{for all $a\in A$ and $b\in B$.}
$$
\end{corollary}
\begin{proof}
A straightforward computation using Proposition~\ref{conparamorfusandomatr} and
Remark~\ref{rho hat por morfismos}.
\end{proof}

\section{Examples}

\subsection{Non-commutative duplicates of finite sets}
Consider the algebra $B=\frac{K[X]}{\langle X^2-X\rangle}$ with the basis $\mathcal{B}=\{1, X\}$. The structure matrices of $1$ and $X$ with respect to $(B,\mathcal{B})$ (see Definition~\ref{matrices de estructura}) are the matrices
$$
[X]_{\mathcal{B}}=\begin{pmatrix} 0 & 0\\ 1 & 1\end{pmatrix} \quad \text{and}\quad [1]_{\mathcal{B}}=\begin{pmatrix} 1 & 0\\ 0 & 1\end{pmatrix}.
$$

Consider a map $\chi\colon A\ot B \longrightarrow B\ot A$. Proceeding as in Subsection~1.2 we determine maps $\gamma_2^1\colon A\to A$ and $\gamma_2^2\colon A\to A$ such that
$$
\chi(a\ot X)= 1\ot\gamma^1_2(a)+X\ot\gamma^2_2(a).
$$

By Corollary~\ref{Coro condiciones 3 y 4} and Proposition~\ref{condiciones 1 y 2} we know that $\chi$ is a twisting map if and only if the maps $\hat{\rho}_{\chi}\colon B^{\op}\longrightarrow M_n(E)$ and $\hat{\phi}_{\chi}\colon A \longrightarrow M_n(A)$ are matrix representations. The associated twisted tensor products $A\ot_{\chi} B$ were studied in \cite{Ci}, where they are called {\em Non-Commutative Duplicates} (Actually they studied twisted tensor products $B\ot_{\chi} A$, which gives the same results, taking the opposite algebras).

The map $\hat{\rho}_B\colon B^{\op} \longrightarrow M_{2}(E)$ is given by the matrices
$$
\hat{\rho}_B(X)=\begin{pmatrix} \gamma^1_2& 0\\ \gamma^2_2 & \gamma^1_2+\gamma^2_2 \end{pmatrix} \quad \text{and}\quad \hat{\rho}_B(1)=\begin{pmatrix} \ide & 0\\ 0 & \ide \end{pmatrix}.
$$
The equalities $\hat{\rho}_B(X)=\hat{\rho}_B(X^2)=\hat{\rho}_B(X)^2$ determine the conditions
\begin{enumerate}

\smallskip

\item $\gamma^1_2\circ\gamma^1_2=\gamma^1_2$,

\smallskip

\item $\gamma^2_2\circ\gamma^1_2+\gamma^1_2\circ\gamma^2_2+\gamma^2_2\circ \gamma^2_2=\gamma^2_2$,

\smallskip

\item $(\gamma^1_2+\gamma^2_2)^2=\gamma^1_2+\gamma^2_2$,

\smallskip

\end{enumerate}
that the maps $\gamma^1_2,\gamma^2_2$ must satisfy in order that~$\hat{\rho}_B$ be a representation.

On the other hand the map $\hat{\phi}_{\chi}$ is given by the matrices
$$
\hat{\phi}_{\chi}(a):= \begin{pmatrix}
a&\gamma_2^1(a)\\
0 &\gamma^2_2(a)
\end{pmatrix} \quad\text{  for $a\in A$.}
$$
A direct computation shows that $\hat{\phi}_{\chi}$ is a representation if and only if the maps $\gamma^1_2$ and $\gamma^2_2$ satisfy
\begin{enumerate}[resume]

\smallskip

\item $\gamma^1_2(1)=0$,

\smallskip

\item $\gamma^2_2(1)=1$,

\smallskip

\item $\gamma^1_2(ab)=a\gamma_2^1(b)+\gamma^1_2(a)\gamma^2_2(b)$,

\smallskip

\item $\gamma^2_2(ab)=\gamma^2_2(a)\gamma^2_2(b)$.

\smallskip

\end{enumerate}
It is easy to to see that (1) and (2) imply (3), while (5) and (6) imply (4). Therefore, the above conditions are satisfied if and only if $f:=\gamma_2^2$ is an endomorphism of $A$ and $\delta:=\gamma_1^2$ is an $(\ide,f)$-derivation satisfying $f=f^2+\delta\circ f +f\circ \delta$ (Compare with \cite{Ci}*{Definition 2.7}).

Finally, the representation $\varphi_{\chi}$ in Corollary~\ref{la otra representacion matricial} is given by the matrices
$$
\varphi_{\chi}(a):= \begin{pmatrix}
a&\gamma_2^1(a)\\
0 &\gamma^2_2(a)
\end{pmatrix} \quad\text{for $a\in A$,}\quad
\varphi_{\chi}(X)=\begin{pmatrix} 0 & 0\\ 1 & 1\end{pmatrix} \quad \text{and}\quad \varphi_{\chi}(1)=\begin{pmatrix} 1 & 0\\ 0 & 1\end{pmatrix}.
$$

\subsection{Factorizations structures with a two-dimensional factor}
Given a polynomial $P(X):= X^2-\alpha X+ \beta\in K[X]$ consider the algebra $B=\frac{K[X]}{\langle P(X)\rangle}$ with the basis $\mathcal{B}=\{1, X\}$. The structure matrices of $1$ and $X$ with respect to $(B,\mathcal{B})$ are the matrices
$$
[X]_{\mathcal{B}}=\begin{pmatrix} 0 & -\beta\\ 1 & \alpha\end{pmatrix} \quad \text{and}\quad [1]_{\mathcal{B}}=\begin{pmatrix} 1 & 0\\ 0 & 1\end{pmatrix}.
$$

Consider a map $\chi\colon A\ot B \longrightarrow B\ot A$. Proceeding as above we determine maps $\gamma_2^1\colon A\to A$ and $\gamma_2^2\colon A\to A$ such that
$$
\chi(a\ot X)= 1\ot\gamma^1_2(a)+X\ot\gamma^2_2(a).
$$

As above, $\chi$ is a twisting map if and only if the maps $\hat{\rho}_{\chi}\colon B^{\op}\longrightarrow M_n(E)$ and $\hat{\phi}_{\chi}\colon A \longrightarrow M_n(A)$ are matrix representations. The associated twisted tensor products $A\ot_{\chi} B$ (respectively $B\ot_{\chi} A$) were studied in \cite{LNC}, where they are called {\em quantum duplicates}. The maps $f$ and $\delta$ considered in that paper correspond to the maps $\gamma_2^2$ and $\gamma_2^1$, respectively.

The map $\hat{\rho}_B\colon B^{\op} \longrightarrow M_{2}(E)$ is given by the matrices
$$
\hat{\rho}_B(X)=\begin{pmatrix} \gamma^1_2& -\beta \gamma_2^2\\ \gamma^2_2 & \gamma^1_2+\alpha \gamma^2_2 \end{pmatrix} \quad \text{and}\quad \hat{\rho}_B(1)=\begin{pmatrix} \ide & 0\\ 0 & \ide \end{pmatrix}.
$$
The equalities $\hat{\rho}_B(X)^2=\hat{\rho}_B(X^2)=\alpha\hat{\rho}_B(X)- \beta \hat{\rho}_B(1) $ determine the conditions
\begin{enumerate}

\smallskip

\item $\gamma^1_2\circ\gamma^1_2 - \beta \gamma^2_2\circ\gamma^2_2=\alpha\gamma^1_2-\beta\ide$,

\smallskip

\item $\gamma^1_2\circ\gamma^2_2+\gamma^2_2\circ\gamma^1_2+\alpha\gamma^2_2\circ\gamma^2_2=\alpha\gamma^2_2$,

\smallskip

\end{enumerate}
that the maps $\gamma^1_2,\gamma^2_2$ must satisfy in order that~$\hat{\rho}_B$ is a representation.

On the other hand the map $\hat{\phi}_{\chi}$ is given by the matrices
$$
\hat{\phi}_{\chi}(a)= \begin{pmatrix}
                            a &\gamma_2^1(a)\\
                            0 &\gamma^2_2(a)
                   \end{pmatrix}\quad\text{for $a\in A$.}
$$
A direct computation shows that $\hat{\phi}_{\chi}$ is a representation if and only if the maps $\gamma^1_2$ and $\gamma^2_2$ satisfy
\begin{enumerate}[resume]

\smallskip

\item $\gamma^2_2(1)=1$,

\smallskip

\item $\gamma^1_2(ab)=a\gamma_2^1(b)+\gamma^1_2(a)\gamma^2_2(b)$,

\smallskip

\item $\gamma^2_2(ab)=\gamma^2_2(a)\gamma^2_2(b)$.

\smallskip

\end{enumerate}
Therefore, the above conditions are satisfied if and only if $f:=\gamma_2^2$ is an endomorphism of $A$ and $\delta:=\gamma_2^1$ is an $(\ide,f)$-derivation satisfying
$$
P(\delta)= \beta f^2 \quad \text{and}\quad f\circ \delta + \delta\circ f = \alpha(f - f^2)
$$
(Compare with \cite{LNC}*{Lemma~1.1}).

Finally the representation $\varphi_{\chi}$ in Corollary~\ref{la otra representacion matricial} is given by the matrices
$$
\varphi_{\chi}(a)= \begin{pmatrix}
                            a &\gamma_2^1(a)\\
                            0 &\gamma^2_2(a)
                   \end{pmatrix}\quad\text{for $a\in A$,}
\quad
\varphi_{\chi}(X)=\begin{pmatrix}
                            0 & -\beta\\
                            1 & \alpha
                  \end{pmatrix}
\quad \text{and}\quad
\varphi_{\chi}(1)=\begin{pmatrix}
                            1 & 0\\
                            0 & 1
                  \end{pmatrix}.
$$

\subsection{Twisting with $K^n$}\label{casoespanholes}
Let $B:=K^n$ and let $\mathcal{B}_2=\{e_1,\dots,e_n\}$ be the canonical basis of $B$. The structure matrices of $e_1,\dots,e_n$
with respect to $(B,\mathcal{B})$ are the matrices $e_{11},\dots, e_{nn}$, where $e_{ii}$ is the matrix with $1$ in the $i$-th entry of the main diagonal and $0$ in the other entries.
Given a twisted tensor product $\chi\colon B\ot A\longrightarrow A\ot B$ we have maps  $\tilde{\gamma}_i^j\colon A\longrightarrow A$ ($1\le i,j \le n$) defined by the equalities
$$
\chi(a\ot e_i)= \sum_{j=1}^ne_1\ot \tilde{\gamma}_i^j(a).
 $$

As above, $\chi$ is a twisting map if and only if the maps $\hat{\rho}_{\chi}\colon B^{\op}\longrightarrow M_n(E)$ and $\hat{\phi}_{\chi}\colon A \longrightarrow M_n(A)$ are matrix representations. The associated twisted tensor products $A\ot_{\chi} B$ were studied in \cite{JLNS}. The maps $E_{ji}$ considered in that paper correspond to the maps $\tilde{\gamma}_j^i$.

The map $\hat{\rho}_B\colon B^{\op} \longrightarrow M_n(E)$ is given by the matrices
$$
\hat{\rho}_B(e_i)=\begin{pmatrix} \tilde{\gamma}^1_i &0  & \dots & 0\\
                                  0 & \tilde{\gamma}^2_i & \dots & 0\\
                                  \vdots & \vdots & \ddots & \vdots \\
                                  0 & 0 & \dots &\tilde{\gamma}^n_i
                                  \end{pmatrix}.
$$
The equalities
$$
\hat{\rho}_B(e_i)^2=\hat{\rho}_B(e_i),\quad\hat{\rho}_B(e_i)\hat{\rho}_B(e_j)=0 \quad \text{and}\quad \hat{\rho}_B(e_1) + \cdots +\hat{\rho}_B(e_n)=\ide
$$
determine the conditions
\begin{enumerate}

\smallskip

\item $\tilde{\gamma}_i^p\circ \tilde{\gamma}_j^p= \delta_{ij}\tilde{\gamma}_i^p$ ($1\le i,j,p\le n$),

\smallskip

\smallskip

\item $\sum_{j=1}^n\tilde{\gamma}^i_j=\ide_A$ for $i=1,\dots,n$,

\smallskip

\end{enumerate}
that the maps $\tilde{\gamma}^p_{ij}$ must satisfy in order that~$\hat{\rho}_B$ is a representation.

On the other hand the map $\hat{\phi}_{\chi}$ is given by the matrices
$$
\hat{\phi}_{\chi}(a)= \begin{pmatrix}
                            \widetilde{\gamma}_1^1(a) & \dots &\widetilde{\gamma}_n^1(a)\\
                            \vdots & \ddots & \vdots \\
                            \widetilde{\gamma}_n^1(a)& \dots &\widetilde{\gamma}^n_n(a)
                   \end{pmatrix} \quad\text{for $a\in A$.}
$$
A direct computation shows that $\hat{\phi}_{\chi}$ is a representation if and only if the maps $\tilde{\gamma}^i_j$ satisfy
\begin{enumerate}[resume]

\smallskip

\item $\tilde{\gamma}^i_j(ab)= \sum_{p=1}^{n}\tilde{\gamma}^i_p(a)\tilde{\gamma}^p_j(b)$ for $i,j=1,\dots,n$,

\smallskip

\item $\tilde{\gamma}^i_j(1_A)=\delta_{ij}1_A$ for $i,j=1,\dots,n$.

\smallskip

\end{enumerate}
Conditions~(1)--(4) correspond to the conditions~(6)--(9) in~\cite{JLNS}.

When $n=2$, we have the isomorphism $\frac{k[X]}{\langle X^2-X \rangle}\simeq K^2$ which sends $X$ to $e_2$. Using Proposition~\ref{morfismo inducido} is straightforward to check that the maps $f$ and $\delta$ in the first example satisfy
$$
f= \tilde{\gamma}^2_2 - \tilde{\gamma}^1_2\quad\text{and}\quad \delta= \tilde{\gamma}^1_2.
$$

The matrix representation $\hat{\phi}_{\chi}$ can be used to define a pair $(Q_{\chi},\mathcal{R}_{\chi})$ where $Q_{\chi}$ is a quiver and
$\mathcal{R}_{\chi}$ is a representation of the quiver $Q_{\chi}$, as follows
\begin{definition}
The Quiver $Q_{\chi}$ is defined as follows:
the set of vertices of $\Gamma_{\chi}$ is $\{v_1,\hdots, v_n \}$. The vertices $ v_i $ and $  v_j$ are joined by an arrow with the source $ v_j $  if and only if $ \gamma_j^i\neq 0$
\end{definition}
\begin{definition}
 Let $Q_{\chi}$be the quiver associated to $\hat{\phi}_{\chi}$. The representation $\mathcal{R}_{\chi}$ of $Q_{\chi}$
 is defined by the family of vector spaces $\{V_i\}_{i\in Q^0_{\chi}}$ and the family of $K$-linear maps $\{f_{\alpha}\}_{\alpha\in Q^1_{\chi}}$, where
  $$
 V_i=A,\quad f_{\alpha}=\gamma^i_j\colon V_j\to V_i,\quad j=s(\alpha),i=t(\alpha)
 $$
\end{definition}
\begin{remark}
 The map $\chi$ is a twisting map if and only if  the pair $(Q_{\chi},\mathcal{R}_{\chi})$ es admissible of order $n$ ( see \cite{JLNS}*{Definition~1.7})
\begin{enumerate}
  \item The conditions~(1)--(2) correspond to splitted condition ( see \cite{JLNS}*{Proposition~1.6})
  \item The condition~(3) corresponds to unital condition ( see \cite{JLNS}*{Proposition~1.6})
  \item The condition~(4) corresponds to factorizable condition ( see \cite{JLNS}*{Proposition~1.6})
\end{enumerate}
\end{remark}

\subsection{Extensions of $K^m$ to $K^{n}$}
Consider a twisting map $\chi\colon A\ot K^n \longrightarrow K^n\ot A$ defined by $\chi(a\ot e_j)=\sum_{i=1}^n e_i\ot\tilde{\gamma}_j^i(a)$ and suppose that exists $m<n$ such that the map
$\Theta\colon A\ot K^{m} \longrightarrow K^{m}\ot A$ defined by $\Theta(a\ot e_j)=\sum_{i=1}^me_i\ot\tilde{\gamma}_j^i(a)$ for $j\in\{1,\hdots,m\}$ is a twisting map.
Write
$$
C^{(1)}_j=\hat{\rho}_{B}(e_j)=\begin{pmatrix} \tilde{\gamma}^1_j &0  & \dots & 0\\
                                  0 & \tilde{\gamma}^2_j & \dots & 0\\
                                  \vdots & \vdots & \ddots & \vdots \\
                                  0 & 0 & \dots &\tilde{\gamma}^m_j
                                  \end{pmatrix},
\quad \text{for} \quad j\in\{m+1,\hdots,n\}.
$$
Then, by proposition~\ref{extension01}, we have $C^{(1)}_j=0$ for $j\in\{m+1,\hdots,n\}$, consequently $\tilde{\gamma}^i_{j}=0$ for $i\in\{1,\hdots,m\}$ and $j\in\{m+1,\hdots,n\}$.
Therefore, the representation $\hat{\phi}_{\chi}$ looks as follows
\begin{equation}\label{ejemplo extension}
\hat{\phi}_{\chi}(a)=\begin{pmatrix}
\hat{\phi}_{\Theta}(a) & 0\\ \Delta(a) & \hat{\phi}_{K^{n-m}}(a)
\end{pmatrix}
\quad\text{para todo $a\in A$,}
\end{equation}
where
$$
\Delta (a) = \begin{pmatrix}
                            \widetilde{\gamma}_1^{m+1}(a) & \dots &\widetilde{\gamma}_m^{m+1}(a)\\
                            \vdots & \ddots & \vdots \\
                            \widetilde{\gamma}_1^n(a)& \dots &\widetilde{\gamma}_m^n(a)
                   \end{pmatrix} \quad%
                   \text{and} \quad
\varphi_{K^{n-m}}(a) = \begin{pmatrix}
                            \widetilde{\gamma}_{m+1}^{m+1}(a) & \dots &\widetilde{\gamma}_n^{m+1}(a)\\
                            \vdots & \ddots & \vdots \\
                            \widetilde{\gamma}_{m+1}^n(a)& \dots &\widetilde{\gamma}_n^n(a)
                   \end{pmatrix},%
$$
for $a\in A$.\\
In particular, the twisted tensor products that come from a Quiver of rank 1 with a cycle of length 2 studied in \cite{JLNS}*{Theorem~4.6}
are extensions of a twisted tensor product generated by the twisting map $\Theta\colon A\ot K^{2} \longrightarrow K^{2}\ot A$.
\newline
On the other hand, when the map $\Upsilon\colon A\ot K^{n-m} \longrightarrow K^{n-m}\ot A$ defined by
$\Upsilon(a\ot e_j)=\sum_{i=m+1}^ne_i\ot\tilde{\gamma}_j^i(a)$ for $j\in\{m+1,\hdots,n\}$ is a twisting map.
Write
$$
B^{(2)}_j=\hat{\rho}_{B}(e_j)=\begin{pmatrix} \tilde{\gamma}^{m+1}_j &0  & \dots & 0\\
                                  0 & \tilde{\gamma}^{m+2}_j & \dots & 0\\
                                  \vdots & \vdots & \ddots & \vdots \\
                                  0 & 0 & \dots &\tilde{\gamma}^n_j
                                  \end{pmatrix},
\quad \text{for} \quad j\in\{1,\hdots,m\}.
$$
Then, by corollary~\ref{extension02}, we have $B^{(2)}_j=0$ for $j\in\{1,\hdots,m\}$, consequently $\tilde{\gamma}^i_{j}=0$ for $i\in\{m+1,\hdots,n\}$ and $j\in\{1,\hdots,m\}$.
Therefore, the representation $\varphi_{\chi}$ looks as follows
\begin{equation}\label{ejemplo extension2}
\varphi_{\chi}(a)=\begin{pmatrix}
\varphi_{\Theta}(a) & 0\\ 0& \varphi_{\Upsilon}(a)
\end{pmatrix}
\quad\text{para todo $a\in A$,}
\end{equation}
The twisted tensor products for which $m=2$ and $\Upsilon$ come from a Quiver of rank 1 without cycles of length 2 studied in
\cite{JLNS}*{Theorem~4.6} are extensions of a twisted tensor product generated by the twisting map $\Theta\colon A\ot K^{2} \longrightarrow K^{2}\ot A$.
By example, the twisting map for which $\varphi_{\Upsilon}=\ide$ where studied in \cite{JLNS}*{Corollary~4.3}.

Now we construct extensions from $K^2$ to $K^3$. Consider a twisting map $\chi\colon K^3\ot K^3 \longrightarrow K^3\ot K^3$ defined by
$$
\chi(a\ot e_j)=e_1\ot\tilde{\gamma}_j^1(a)+e_2\ot\tilde{\gamma}_j^2(a)+e_3\ot\tilde{\gamma}_j^3(a)\quad j\in\{1,2,3\}
$$
and suppose that the map
$\Theta\colon K^3\ot K^2 \longrightarrow K^2\ot K^3$ defined by
$$\Theta(a\ot e_j)=e_1\ot\tilde{\gamma}_j^1(a)+e_2\ot\tilde{\gamma}_j^2(a)\quad j\in\{1,2\}$$
is a twisting map. Then, by Proposition \ref{teorema fundamental sobre extensiones}, we know that 
$\tilde{\gamma}_3^1=\tilde{\gamma}_3^2=0$. 

\subsection{Noncommutative truncated polynomial extensions}
Let $B:=\frac{K[Y]}{\langle Y^n \rangle}$ with the basis $\mathcal{B}=\{1,Y,Y^2,\hdots,Y^{n-1}\}$.
The structure matrices of $1$ and  $Y$ with respect to $(B,\mathcal{B})$ are the matrices
$$
[Y]_{\mathcal{B}}=
\begin{pmatrix}
0& 0&\hdots&0&0\\
1& 0&\hdots&0&0\\
0& 1&\hdots&0&0\\
\vdots &\vdots &\ddots&\vdots&\vdots\\
0& 0&\hdots&1&0\\
\end{pmatrix}
 \quad \text{and}\quad
[1]_{\mathcal{B}}=
\begin{pmatrix}
1& 0&0&\hdots&0\\
0& 1&0&\hdots&0\\
0& 0&1&\hdots&0\\
\vdots &\vdots &\vdots&\ddots&\vdots\\
0& 0&0&\hdots&1\\
\end{pmatrix}.
$$

Given a map $\chi\colon B\ot A\longrightarrow A\ot B$ we have unique maps  $\tilde{\gamma}_i^j\colon A\longrightarrow A$ ($0\le i,j \le n-1$) such that
$$
\chi(a\ot Y^i)= \sum_{j=0}^{n-1}Y^j\ot \tilde{\gamma}_i^j(a).
$$

We know that $\chi$ is a twisting map if and only if $\hat{\rho}_{\chi}\colon B^{\op}\longrightarrow M_n(E)$ and $\hat{\phi}_{\chi}\colon A \longrightarrow M_n(A)$ are matrix representations. The associated twisted tensor products $A\ot_{\chi} B$ were studied in \cite{GGV}. The maps $\gamma_i^j\colon A^{\op} \longrightarrow A^{\op}$ considered in that paper are related to the maps $\tilde{\gamma}_j^i\colon A\to A$ by the equalities $\gamma_i^j(a^{\op})= \tilde{\gamma}_j^i(a)$ for $0\le i,j < n$ and all $a\in A$.

The map $\hat{\rho}_B\colon B^{\op} \longrightarrow M_n(E)$ is given by the matrices
$$
\hat{\rho}_B(Y^i)=\begin{pmatrix}
\tilde{\gamma}^0_i& 0&0&\hdots&0\\
\tilde{\gamma}^1_i & \tilde{\gamma}^0_i&0&\hdots&0\\
\tilde{\gamma}^2_i & \tilde{\gamma}^1_i&\tilde{\gamma}^0_i&\hdots&0\\
\vdots &\vdots &\vdots&\ddots&\vdots\\
\tilde{\gamma}^{n-1}_i & \tilde{\gamma}^{n-2}_i&\tilde{\gamma}^{n-3}_i&\hdots&\tilde{\gamma}^0_i
\end{pmatrix}.
$$
The equalities
$$
\hat{\rho}_B(Y^r)=\hat{\rho}_B(Y^{r-i})\hat{\rho}_B(Y^i),\quad\hat{\rho}_B(Y^n)=0\quad \text{and}\quad \hat{\rho}_B(1)=\ide
$$
determine the conditions
\begin{enumerate}

\smallskip

\item $\tilde{\gamma}_0^j= \delta_{0j}\ide$,

\smallskip

\item $\tilde{\gamma}^j_r=\sum_{l=0}^j\tilde{\gamma}^{j-l}_{r-i}\circ\tilde{\gamma}^l_i$ for $j<n$, $1<r<n$ and $0<i<r$,

\smallskip

\item $\sum_{l=0}^j\tilde{\gamma}^{j-l}_{n-i}\circ\tilde{\gamma}^l_i=0$ for $j<n$ and $0<i<n$,

\end{enumerate}
that the maps $\tilde{\gamma}^p_{ij}$ must satisfy in order that~$\hat{\rho}_B$ is a representation.

On the other hand the map $\hat{\phi}_{\chi}$ is given by the matrices
$$
\hat{\phi}_{\chi}(a)= \begin{pmatrix}
                            \tilde{\gamma}_0^0(a) & \dots &\tilde{\gamma}_{n-1}^0(a)\\
                            \vdots & \ddots & \vdots \\
                            \tilde{\gamma}^{n-1}_0(a)& \dots &\tilde{\gamma}^{n-1}_{n-1}(a)
                   \end{pmatrix} \quad\text{for $a\in A$.}
$$
A direct computation shows that $\hat{\phi}_{\chi}$ is a representation if and only if the maps $\tilde{\gamma}^i_j$ satisfy
\begin{enumerate}[resume]

\smallskip

\item $\tilde{\gamma}^i_j(ab)= \sum_{p=1}^{n}\tilde{\gamma}^i_p(a)\tilde{\gamma}^p_j(b)$ for $i,j=0,\dots,n-1$,

\smallskip

\item $\tilde{\gamma}^i_j(1_A)=\delta_{ij}1_A$ for $i,j=0,\dots,n-1$.

\smallskip

\end{enumerate}
Conditions~(1)--(5) correspond to the conditions~(2)(a)--(d) in~\cite{GGV}*{Proposition 1.2} and to the assumption $\gamma^r_j=0$ for $r\ge n$ made in that paper.


\begin{bibdiv}
\begin{biblist}

\bib{CSV}{article}{
   author={Cap, Andreas},
   author={Schichl, Hermann},
   author={Van{\v{z}}ura, Ji{\v{r}}{\'{\i}}},
   title={On twisted tensor products of algebras},
   journal={Comm. Algebra},
   volume={23},
   date={1995},
   number={12},
   pages={4701--4735},
   issn={0092-7872},
   review={\MR{1352565 (96k:16039)}},
   doi={10.1080/00927879508825496},
}

\bib{Ca}{article}{
   author={Cartier, P},
   title={Produits tensoriels tordus},
   journal={Expos\'{e} au S\'{e}minaire des groupes quantiques de l' \'{E}cole Normale
Sup\'{e}rieure, Paris},
   date={1991-1992},
}

\bib{Ci}{article}{
   author={Cibils, Claude},
   title={Non-commutative duplicates of finite sets},
   journal={J. Algebra Appl.},
   volume={5},
   date={2006},
   number={3},
   pages={361--377},
   issn={0219-4988},
   review={\MR{2235816 (2007d:16020)}},
   doi={10.1142/S0219498806001776},
}

\bib{GG}{article}{
   author={Guccione, Jorge A.},
   author={Guccione, Juan J.},
   title={Hochschild homology of twisted tensor products},
   journal={$K$-Theory},
   volume={18},
   date={1999},
   number={4},
   pages={363--400},
   issn={0920-3036},
   review={\MR{1738899 (2001a:16016)}},
   doi={10.1023/A:1007890230081},
}

\bib{GGV}{article}{
   author={Guccione, Jorge A.},
   author={Guccione, Juan J.},
   author={Valqui, Christian},
   title={Non commutative truncated polynomial extensions},
   journal={J. Pure Appl. Algebra},
   volume={216},
   date={2012},
   number={11},
   pages={2315--2337},
   issn={0022-4049},
   review={\MR{2927170}},
   doi={10.1016/j.jpaa.2012.01.021},
}

\bib{GGVTP}{article}{
   author={Guccione, Jorge A.},
   author={Guccione, Juan J.},
   author={Valqui, Christian},
   title={Twisted planes},
   journal={Comm. Algebra},
   volume={38},
   date={2010},
   number={5},
   pages={1930--1956},
   issn={0092-7872},
   review={\MR{2642035 (2011h:16031)}},
   doi={10.1080/00927870903023105},
}

\bib{JLNS}{article}{
   author={Jara, P.},
   author={L{\'o}pez Pe{\~n}a, J.},
   author={Navarro, G.},
   author={{\c{S}}tefan, D.},
   title={On the classification of twisting maps between $K^n$ and $K^m$},
   journal={Algebr. Represent. Theory},
   volume={14},
   date={2011},
   number={5},
   pages={869--895},
   issn={1386-923X},
   review={\MR{2832263 (2012g:16053)}},
   doi={10.1007/s10468-010-9222-x},
}

\bib{LNC}{article}{
   author={Cortadellas, {\'O}scar},
   author={L{\'o}pez Pe{\~n}a, Javier},
   author={Navarro, Gabriel},
   title={Factorization structures with a two-dimensional factor},
   journal={J. Lond. Math. Soc. (2)},
   volume={81},
   date={2010},
   number={1},
   pages={1--23},
   issn={0024-6107},
   review={\MR{2580451 (2010m:16035)}},
   doi={10.1112/jlms/jdp055},
}

\bib{Ma}{article}{
   author={Majid, Shahn},
   title={Physics for algebraists: noncommutative and noncocommutative Hopf
   algebras by a bicrossproduct construction},
   journal={J. Algebra},
   volume={130},
   date={1990},
   number={1},
   pages={17--64},
   issn={0021-8693},
   review={\MR{1045735 (91j:16050)}},
   doi={10.1016/0021-8693(90)90099-A},
}

\bib{Tam}{article}{
   author={Tambara, D.},
   title={The coendomorphism bialgebra of an algebra},
   journal={J. Fac. Sci. Univ. Tokyo Sect. IA Math.},
   volume={37},
   date={1990},
   number={2},
   pages={425--456},
   issn={0040-8980},
   review={\MR{1071429 (91f:16048)}},
}

\bib{VDVK}{article}{
   author={Van Daele, A.},
   author={Van Keer, S.},
   title={The Yang-Baxter and pentagon equation},
   journal={Compositio Math.},
   volume={91},
   date={1994},
   number={2},
   pages={201--221},
   issn={0010-437X},
   review={\MR{1273649 (95c:16052)}},
}

\end{biblist}
\end{bibdiv}
\end{document}